%% file: main.tex
\documentclass[a4paper,12pt]{article}

\input{pre}

\title{Numerical Analysis of Finite Dimensional Approximations in Finite Temperature DFT
}
\author{
Ge Xu\footnote{{\tt gexu@mail.bnu.edu.cn}.
School of Mathematical Sciences, Beijing Normal University, Beijing 100875, China.
}, 
~ 
Huajie Chen\footnote{Corresponding. {\tt chen.huajie@bnu.edu.cn}.
School of Mathematical Sciences, Beijing Normal University, Beijing 100875, China.
} 
~ and ~ Xingyu Gao\footnote{{\tt gao\_xingyu@iapcm.ac.cn}.
National Key Laboratory of Computational Physics, Institute of Applied Physics and Computational Mathematics, Beijing 100088, China.
}
}
\date{}

\begin{document}
\maketitle

\begin{abstract}
In this paper, we study numerical approximations of the ground states in finite temperature density functional theory. 
We formulate the problem with respect to the density matrices and justify the convergence of the finite dimensional approximations. 
Moreover, we provide an optimal {\it a priori} error estimate under some mild assumptions and present some numerical experiments to support the theory.

\vskip 0.5cm


\vskip 0.3cm


\end{abstract}

\section{Introduction}
\label{sec:intro}
\setcounter{equation}{0}

Density functional theory (DFT) \cite{hohenberg64,kohn65} has been the most widely used method in electronic structure calculations, which achieves the best compromise between accuracy and computational cost among different approaches \cite{CancesFriesecke23,lin19,martin05,ParrYang94}.
While the standard DFT models are in principle for systems at zero temperature, a finite temperature DFT model was proposed \cite{mermin65} for systems where the temperature effects on electrons are not negligible.
The finite temperature DFT can be viewed as an extension of DFT by including the electron entropy into the total energy, which considers the electrons within a canonical ensemble \cite{ParrYang94}.
It has not only been successfully applied to many practical simulations (see e.g. \cite{kormann12,stegailov10}), but also used to resolve the ``charge sloshing" phenomenon in self-consistent field iterations for metallic systems by smearing the the integer occupation numbers into fractional ones (see e.g. \cite{liu15,motamarri13}).

The finite temperature DFT model is usually formulated by the so-called Mermin-Kohn-Sham (MKS) equation \cite{mermin65,martin05,parr89}:
Find the chemical potential $\mu\in\R$, the orbitals and occupation numbers $(\phi_i,\lambda_i)\in H^1(\R^3)\times\R~(i=1,2,\cdots)$ satisfying
\begin{eqnarray}
\label{intro_MKS}
\begin{cases}
\displaystyle
\left(-\frac{1}{2}\Delta + v_{\rm ext} + v_{\rm H}(\rho) + v_{\rm xc}(\rho) \right) \phi_i = \lambda_i\phi_i  \qquad & ({\rm a})
\\[10pt]
\displaystyle
\int_{\mathbb{R}^3}\phi_i\phi_j = \delta_{ij} & ({\rm b})
\\[10pt]
\displaystyle
\rho = \sum_i f_i|\phi_i|^2 & ({\rm c})
\\[10pt]
\displaystyle
f_i = \left( 1+e^{(\lambda_i-\mu)/(k_{\rm B}T)} \right)^{-1} & ({\rm d})
\\[10pt]
\displaystyle
\sum_i f_i = N & ({\rm e})
\end{cases}
\end{eqnarray}
with $k_{\rm B}$ the Boltzmann constant and $T$ the electron temperature of the system.
The chemical potential $\mu$ is actually determined by the constraint (d) that the sum of occupation numbers equals the electron number.
In equation (a), the potential includes the external potential $v_{\rm ext}$ representing the nuclear attraction, the Hartree potential $V_{\rm H}(\rho)=\int_{\mathbb{R}^3}\frac{\rho(r)}{|\cdot-r|}\dd r$ representing the mean-field electron repulsion, and the exchange-correlation potential $v_{\rm xc}(\rho)$ \cite{martin05}.
The Mermin-Kohn-Sham equation is a nonlinear eigenvalue problem, as the operator depends on both the eigenvalues and eigenfunctions to be solved.
This is usually solved by a self-consistent field (SCF) iteration algorithm in practice.
We note that the standard Kohn-Sham equations will be recovered at the zero temperature limit, which gives the occupation numbers $f_i=1$ when $i\leq N$ and $f_i=0$ when $i>N$ under the gap condition $\lambda_N<\lambda_{N+1}$.

Validation of numerical results is a fundamental issue, particularly the discretization error resulting from the choice of a finite basis set.
For standard DFT models like orbital-free and Kohn-Sham models, there are many works on the numerical analysis in the past decade \cite{cances10,cances12,chen13,chen10,chen11,gavini07,lang10,sur10,zhou07}.
The extension of the existing analysis to the finite temperature DFT model presents additional complexity, as it requires the consideration of a significantly bigger set of orbitals and occupation numbers.
To the best of our knowledge, there is very limited work devoted to its mathematical and numerical analysis.
In \cite{prodan03}, the existence and uniqueness of solutions for the MKS equations were proved by using the Banach's fixed point theorem under the condition that the coupling constant is sufficiently small.
In \cite{cornean08}, the existence and uniqueness were justified with the assumption the absence of exchange correlation potential, by using the Schauder's fixed point theorem.
In \cite{cui2023}, a similar problem, the Schr\"{o}dinger--Possion equation was considered and the a priori error estimate was derived.
We mention that for general MKS equation, the convergence analysis and error estimates under the energy norm are still missing.

The purpose of this work is to justify the convergence of the finite dimensional approximations for finite temperature DFT models and provide a sharp error estimate. 
Unlike the existing works, we formulate the problem with respect to the electron density matrix, which includes all information of the solutions and at the same time avoids handling infinitly many number of orbitals and occupation numbers.
With the density matrix formulation, the ground state of the system can be obtained by solving a variational problem with respect to the density matrix.
We then prove the convergence by using a variational calculus argument, and derive an optimal {\it a priori} error estimate by using the inverse function theorem. 
The results of this work can support a fairly wide range of numerical methods rather than a particular discretization scheme.

\vskip 0.1cm

{\bf Outline}.
The rest of this paper is organized as follows.
In Section \ref{sec:model}, we formulate the finite temperature DFT model as a variational problem with respect to the one-electron density matrix. 
In Section \ref{sec:gs}, we study the free energy functional and give the optimality conditions for the varaitional problem. 
In Sections \ref{sec:numerical_anal}, we show the convergence of finite dimensional approximations and further derive an optimal a priori error estimate under reasonable assumptions.
In Section \ref{sec:numerics}, we present several numerical experiments by plane-wave methods to support our theory.
In Section \ref{sec_conclusion}, we give some concluding remarks and future perspectives.

\vskip 0.1cm

{\bf Notations}.
Throughout this paper, we will use the Dirac bra-ket notation \cite{dirac39}, which defines the ``ket" vector $|\cdot\rangle$ and its conjugate transpose ``bra" vector $\langle\cdot|$.
This has been widely used in quantum mechanics.

For a Banach space $X$, we denote its the dual space by $X^*$.
We denote by $\mathcal{B}(X,Y)$ the space of bounded linear operators from a Banach space $X$ to a Banach space $Y$.
For a Hilbert space $\mathcal{H}$, we denote by $\mathcal{S}(\mathcal{H})$ the space of self-adjoint operators on the Hilbert space $\mathcal{H}$.

We use the standard notations for $L^p$-spaces and $W^{s,p}$-Sobolev spaces, with their associated norms and semi-norms.
We will use $(\cdot,\cdot)$ for the standard $L^2$ inner product.
For $p=2$, we denote by $H^s(\Omega)=W^{s,2}(\Omega)$, and by $H^1_{\#}(\Omega)\subset H^1(\Omega)$ the subspace with appropriate boundary conditions, 
for example, the zero Direichlet boundary condition or the periodic boundary condition.

Let $\mathfrak{S}^1(\Omega)$ be the space of trace class operators and $\mathfrak{S}^2(\Omega)$ be the space of Hilbert-Schmidt operators on $L^2(\Omega)$.
We further define the following Banach space (see \cite{cances14})
\begin{eqnarray}
\label{space}
\mathfrak{S}^{1,1}(\Omega) := \Big\{A\in\mathfrak{S}^1(\Omega)\cap\mathcal{S}\big(L^2(\Omega)\big): |\nabla|A|\nabla|
\in\mathfrak{S}^1(\Omega)\Big\}
\end{eqnarray}
with the induced norm
\begin{eqnarray}
\label{norm}
\|A\|_{\mathfrak{S}^{1,1}(\Omega)} := \|A\|_{\mathfrak{S}^1(\Omega)} + \big\||\nabla| A|\nabla| \big\|_{\mathfrak{S}^1(\Omega)}
= {\rm Tr}\big(|A|\big) + {\rm Tr}\big(|\nabla|A|\nabla|\big) .
\end{eqnarray}

We will use $C$ to denote a generic positive constant that may change from one line to the next.
The dependencies of $C$ on model parameters will normally be clear from the context or stated explicitly.

\section{Free energy in density matrix formulation}
\label{sec:model}
\setcounter{equation}{0}

In the MKS equation \eqref{intro_MKS}, there are infinitely many number of orbitals and corresponding occupation numbers, which bring significant difficulty to the model formulation and numerical analysis.
Instead of the orbitals and occupation numbers, the one-electron density matrix contains all the information, while at the same time provides a much elegant representation of the electron states.
In this work, we shall first reformulate the problem by a variational problem with respect to the one-electron density matrices. 

Physically the system under consideration is set over $\mathbb{R}^3$, but in practicical calculations it is commonly to restrict the system to a bounded domain.
Let $\Omega\subset\mathbb{R}^3$ be a polytypic bounded domain. 
The electron states will then be restricted to the domain with appropriate boundary conditions.

Let $\mathcal{K}_N$ be the set of one-body reduced density matrices
\begin{eqnarray}
\label{set_gamma}
\mathcal{K}_N := \Big\{ \varGamma\in \mathfrak{S}^{1,1}(\Omega): ~0\leq\varGamma\leq 1,~{\rm Tr}(\varGamma)=N \Big\},
\end{eqnarray}
where the space $\mathfrak{S}^{1,1}(\Omega)$ is given in \eqref{space}, and $0\leq\varGamma\leq 1$ stands for $0\leq \big\langle\phi|\varGamma|\phi\big\rangle \leq\|\phi\|_{L^2(\Omega)}^2$ for any $\phi\in L^2(\Omega)$. 
The inequality constraint of $\varGamma$ means that the spectrum of $\varGamma$ should lie in $[0,1]$.
For $\varGamma \in \mathcal{K}_N$, the corresponding electron density associated to $\varGamma$ is
\begin{eqnarray}
\label{rho}
\rho_{\varGamma}(x) := \gamma(x,x),
\end{eqnarray}
where $\gamma(x,y)$ is the kernel of the operator $\varGamma$.
The density matrix $\varGamma$ can be formulated in the following form
\begin{eqnarray}
\label{gamma_dirac}
\varGamma = \sum_i f_i|\phi_i\rangle\langle\phi_i|,
\end{eqnarray}
where the orbitals $\{\phi_i\}_{i=1}^{\infty}$ and the occupation numbers $\{f_i\}_{i=1}^{\infty}$ satisfy 
\begin{eqnarray*}
\label{phi_space}
\phi_i\in H_{\#}^1(\Omega),~~\int_{\Omega}\phi_i\phi_j=\delta_{ij},~~f_i\in[0,1],~~\sum_{i=1}^{\infty}f_i=N~~{\rm and}~~\sum_{i=1}^{\infty}f_i\int_{\Omega}|\nabla \phi_i|^2<\infty.
\end{eqnarray*}

The Helmholtz free energy of the system is given by the sum of energy and entropy.
The energy $\E(\varGamma)$ is given by
\begin{eqnarray}
\label{energy_gamma}
\E(\varGamma) = {\rm Tr}\left(-\frac{1}{2}\Delta\varGamma\right)
+ \int_{\Omega} v_{\rm ext}\rho_{\varGamma}
+ \frac{1}{2}\int_{\Omega}\int_{\Omega}\frac{\rho_{\varGamma}(x)\rho_{\varGamma}(x^{\prime})}{|x-x^{\prime}|}\dd x\dd x^{\prime}
+ \int_{\Omega} e_{\rm xc}(\rho_{\varGamma}),
\end{eqnarray}
where $v_{\rm ext}$ represents the external potential and $e_{\rm xc}(\rho_{\varGamma})$ is the so-called exchange-correlation energy per volume with electron density $\rho_{\varGamma}$ \cite{martin05}.
The entropy $\mathcal{S}(\varGamma)$ is given by
\begin{eqnarray}
\label{entropy}
\mathcal{S}(\varGamma) = \beta^{-1}{\rm Tr} \big (\varGamma\ln\left(\varGamma\right) + \left(I-\varGamma\right)\ln\left(I-\varGamma\right) \big ), 
\end{eqnarray}
where $I$ denotes the identity operator, $\beta= k_B T$ is a constant with $k_B$ the Boltzmann constant and $T$ the temperature. 
Then the free energy functional $\F:\mathcal{K}_N\rightarrow\R$ is given by
\begin{eqnarray}
\label{free_energy_gamma}
\F(\varGamma) := \E(\varGamma) 
+ \mathcal{S}(\varGamma).
\end{eqnarray}

We will need the following assumptions for our convergence analysis.
The assumptions are similar to those in \cite{cances10, chen13}, which can be satisfied for most of the DFT models with local-density approximations (LDA) \cite{martin05}.

\vskip 0.2cm

\noindent
{\bf A1}.
$v_{\rm ext}\in L^{2}\left(\Omega\right)$.

\vskip 0.2cm

\noindent
{\bf A2}.
There exists a constant $c_0>0$, such that
$|e_{\rm xc}(t)|\leq c_0\big(t^{\frac{4}{3}}+1\big)$ for any $t\geq 0$.  

\vskip 0.2cm

\noindent
{\bf A3}.
$e_{\rm xc}\in C^3$ and there exist constants $c_1,c_2>0$, and $p_1 \in[0,2]$, $p_2\in(0,1]$ such that
\begin{align*}
\left|e_{\rm xc}^{\prime}(t)\right| + \left|t e_{\rm xc}^{\prime \prime}(t)\right| \leq c_1\big(1+t^{p_1}\big)
\quad{\rm and}\quad
\left|e_{\rm xc}^{\prime \prime}(t)\right| + \left|t e_{\rm xc}^{\prime \prime \prime}(t)\right| \leq c_2\big(1+t^{p_2-1}\big)
\qquad \forall~t\geq 0.
\end{align*}

\vskip 0.2cm

With the aforementioned assumptions, we can establish both the coercivity and weak lower semicontinuity of the free energy functional $\F\left(\varGamma\right)$ in the following two lemmas.

\begin{lemma}[Coercivity]
\label{lemma:coercivity}
If the assumptions {\bf A1}-{\bf A2} are satisfied, then there exist positive constants $a$ and $b$ such that 
\begin{eqnarray}
\label{bounded_below}
\F\left(\varGamma\right) \geq a\|\varGamma\|_{\mathfrak{S}^{1,1}\left(\Omega\right)} - b
\qquad \forall~ \varGamma \in \mathcal{K}_N.
\end{eqnarray}
\end{lemma}

\begin{proof}
We first consider the energy part \eqref{energy_gamma}.
Using $\rho=\sum f_i\left|\phi_i\right|^2$ and the Cauchy-Schwarz inequality, we have
\begin{align*}
|\nabla \sqrt{\rho}| & = \frac{1}{2} \frac{|\nabla \rho|}{\sqrt{\rho}} 
\leq \frac{1}{\sqrt{\rho}} \sum_{i=1}^{\infty} f_i\left|\phi_i\right|\left|\nabla \phi_i\right|
= \left(\sum_{i=1}^{\infty} \sqrt{f_i}\left|\nabla \phi_i\right| \frac{\sqrt{f_i}\left|\phi_i\right|}{\sqrt{\rho}}\right)^2 
\\[1ex] 
& \leq \left(\sum_{i=1}^{\infty} f_i\left|\nabla \phi_i\right|^2\right)^{\frac{1}{2}} \left(\frac{1}{\rho} \sum_{i=1}^{\infty} f_i\left|\phi_i\right|^2\right)^{\frac{1}{2}} 
= \left(\sum_{i=1}^{\infty} f_i\left|\nabla \phi_i\right|^2\right)^{\frac{1}{2}} .
\end{align*}
This implies
\begin{eqnarray}
\label{coer-kinetic}
{\rm Tr}\left(-\Delta\varGamma\right) = {\rm Tr}\big(|\nabla| \varGamma |\nabla|\big) = \sum_{i=1}^{\infty} f_i \int_{\mathbb{R}^3}\left|\nabla \phi_i\right|^2
\geq C \int_{\mathbb{R}^3}|\nabla \sqrt{\rho}|^2.
\end{eqnarray}
Then by using the assumptions {\bf A1}-{\bf A2} and the same arguments as that in \cite[Proposition 3.1]{chen10}, we have that there exists a constant $b_1>0$ such that
\begin{eqnarray}
\label{proof-coer-a}
\E(\varGamma) \geq \frac{1}{4} {\rm Tr}\left(-\Delta\varGamma\right)  - b_1 .
\end{eqnarray}

We next want to show that the entropy part \eqref{entropy} is also bounded from below by the kinetic term.
Let $(\lambda_k,\psi_k)\in\R\times H_{\#}^1(\Omega)$ be the eigenpairs of $-\Delta$ on $\Omega$.
By using \eqref{gamma_dirac} and the fact that $\{\psi_k\}$ form a complete basis set of $H_{\#}^1(\Omega)$, we have that
\begin{eqnarray}
\label{tr_Lap}
{\rm Tr}\left(-\Delta\varGamma\right) = 
\sum_k\langle\psi_k|-\Delta \varGamma|\psi_k\rangle = 
\sum_k \sum_j f_j\langle\psi_k|-\Delta|\phi_j\rangle\langle\phi_j|\psi_k\rangle
\geq \sum_k f_k\lambda_k.
\end{eqnarray}
We then resort to the result for asymptotic distribution of Laplacian eigenvalues \cite{reed03}.
Let $N(\lambda) := {\rm dim} P_{(-\infty, \lambda)}(-\Delta)$ denote the number of Laplacian eigenvalues less than $\lambda$. 
It is shown in \cite{reed03} that as $\lambda \to \infty$, $N(\lambda) \leq C_{\Omega}\lambda^{\frac{3}{2}}$. 
This together with \eqref{tr_Lap} implies that the kinetic term has a lower bound
\begin{eqnarray}
\label{proof:entropy:a}
{\rm Tr}(-\Delta\varGamma)
\geq C\sum_k f_k k^{\frac{2}{3}} .
\end{eqnarray}
To estimate the first part ${\rm Tr}\big(\varGamma\ln(\varGamma)\big) = \sum_{k=1}^{\infty}f_k\ln f_k$ in the entropy \eqref{entropy}, we let $\eta\in (0,\frac{1}{4})$ and $I_{\eta}:=\{k\in\mathbb{N}_+:f_k\leq \frac{1}{k^{1+\eta}}\}$.
Then by using the monotonicity of the function $-x\ln x$ on $[0,1/e]$, we have that for any $\delta>0$, there exists a constant $c_{\delta}>0$ such that
\begin{align}
\label{proof:entropy:b}
\nonumber
-\sum_k f_k\ln f_k
& = -\sum_{k\in I_{\eta}} f_k\ln f_k - \sum_{k\in \mathbb{N}_+\backslash I_{\eta}} f_k\ln f_k
\leq \sum_{k\in I_{\eta}} \frac{1+\eta}{k^{1+\eta}}\ln k + \sum_{k\in \mathbb{N}_+\backslash I_{\eta}} (1+\eta)f_k\ln k
\\[1ex]
& \leq C_{\eta} + \delta \sum_{k\in \mathbb{N}_+\backslash I_{\eta}} f_k k^{\frac{2}{3}} - c_{\delta} .
\end{align}
For the second part ${\rm Tr}\big((I-\varGamma) \ln(I-\varGamma)\big)$ in the entropy \eqref{entropy}, we let $J:=\{k\in\mathbb{N}_+:f_k>\frac{1}{2}\}$.
We have $-(1-f_k)\ln (1-f_k)<\frac{1}{e}$ for any $k\in J$ 
and $-(1-f_k)\ln (1-f_k)\leq-f_k\ln f_k$ for any $k\in \mathbb{N}_+\backslash J$.
Note that $\sum_k f_k = N$ implies that $|J|\leq 2N$ and hence
\begin{align}
\label{proof:entropy:c}
-\sum_{k=1}^\infty (1-f_k)\ln (1-f_k)
& = \sum_{k\in J}\frac{1}{e} - \sum_{k\in \mathbb{N}_+\backslash J} (1-f_k)\ln (1-f_k)
\leq C - \sum_{k\in \mathbb{N}_+\backslash J} f_k\ln (f_k) .
\end{align}
By combing \eqref{proof:entropy:a}, \eqref{proof:entropy:b} and \eqref{proof:entropy:c}, we have that for any $\delta>0$, there exists a positive constant $b_{\delta}$ such that
\begin{eqnarray}
\label{proof-coer-b}
\mathcal{S}(\varGamma) \geq -\delta{\rm Tr}(-\Delta\varGamma)-b_{\delta}.
\end{eqnarray}
Then we can conclude \eqref{bounded_below} by using \eqref{proof-coer-a} and \eqref{proof-coer-b}, which completes the proof.
\end{proof}

\begin{lemma}[Weak lower semicontinuity]
\label{lemma:semicontinuity}
Let $\varGamma\in \mathfrak{S}^{1,1}\left(\Omega\right)$, and $\{\varGamma_n\}_{n\in \mathbf{N}}\subset\mathfrak{S}^{1,1}\left(\Omega\right)$ be a sequence that converges to $\varGamma$ in the weak-* topology.
If the assumptions {\bf A1}-{\bf A2} are satisfied, then
\begin{eqnarray}
\label{lemma-2}
\F(\varGamma) \leq\liminf _{n \rightarrow \infty} \F(\varGamma_n).
\end{eqnarray}
\end{lemma}

\begin{proof}
As by \eqref{coer-kinetic}, $\left\{\sqrt{\rho_{\varGamma_n}}\right\}_{n \in \mathbb{N}}$ is bounded in $H_{\#}^1(\Omega)$, it follows that $\sqrt{\rho_{\varGamma_{n}}}$ converges to $\xi$ weakly in $H_{\#}^{1}\left(\Omega\right)$, for some $\xi \in H_{\#}^1(\Omega)$. 
According to Embedding theorem, we assert that $\{\sqrt{\rho_{\varGamma_{n}}}\}_{n \in \mathbb{N}}$ converges strongly to $\xi$ in $L^{p}\left(\Omega\right)$ for all $2 \leqslant p<6$.
Then $\{\rho_{\varGamma_{n}}\}_{n \in \mathbb{N}}$ converges to $\xi^2$ strongly in $L^{p}\left(\Omega\right)$ for all $1 \leqslant p<3$. 
The problem is that whether $\rho_{\Gamma}$ coincides with $\xi^2$. 
Due to the uniqueness of weak limit, it is only necessary to show that $\rho_{\Gamma_n}$ converges to $\rho_{\Gamma}$ weakly in $L^2(\Omega)$. 


Note that for any $W \in C_c^{\infty}\left(\Omega\right)$, the operator $(I+|\nabla|)^{-1} W(I+|\nabla|)^{-1}$ is compact on $L^{2}\left(\Omega\right)$.
As $\{\varGamma_n\}_{n\in \mathbb{N}}$ converge to $\varGamma$ for the weak-* topology in $\mathfrak{S}^{1,1}\left(\Omega\right)$, we have that for any compact operator $K$ on $L^2(\Omega)$,
\begin{eqnarray*}
\lim_{n \rightarrow \infty} \operatorname{Tr}\left(\varGamma_{n}K\right)=\operatorname{Tr}(\varGamma K)
\quad{\rm and}\quad
\lim_{n \rightarrow \infty}\operatorname{Tr}\left(|\nabla|\varGamma_{n}|\nabla| K\right)=\operatorname{Tr}(|\nabla| \varGamma|\nabla|K).
\end{eqnarray*}
Therefore, we have
\begin{align*} 
\lim_{n \rightarrow \infty}
\int_{\Omega} \rho_{\varGamma_{n}} W
&= \lim_{n \rightarrow \infty}
\operatorname{Tr}\left(\varGamma_{n}W\right)
\\
& =\lim_{n \rightarrow \infty}
\operatorname{Tr}\left((I+|\nabla|) \varGamma_{n}(I+|\nabla|)(I+|\nabla|)^{-1} W(I+|\nabla|)^{-1}\right)
\\
& = \operatorname{Tr}\left((I+|\nabla|) \varGamma(I+|\nabla|)(I+|\nabla|)^{-1} W(I+|\nabla|)^{-1}\right)
\\
& = \operatorname{Tr}(\varGamma W)=\int_{\Omega} \rho_{\varGamma} W.
\end{align*}
This implies that $\{\rho_{\varGamma_{n}}\}_{n \in \mathbb{N}}$ converges to $\rho_{\varGamma}$ in $\big(C_c^{\infty}\left(\Omega\right)\big)'$ and hence  $\rho_{\Gamma}$ indeed coincides with $\xi^2$, so we get $\rho_{\Gamma_n} \rightarrow \rho_{\Gamma}$ in $L^p(\Omega)$ for any $1 \leq p<3$.
By using the assumptions {\bf A1}, {\bf A2} and the same argument as that in \cite[Proposition 3.2]{chen10}, we have
\begin{multline}
\label{proof-ex}
\qquad
\lim _{n \rightarrow \infty} 
\left( \int_{\Omega} v_{\rm ext} \rho_{\varGamma_{n}} + \frac{1}{2}\int_{\Omega}\int_{\Omega}\frac{\rho_{\varGamma_n}(x)\rho_{\varGamma_n}(x^{\prime})}{|x-x^{\prime}|}\dd x\dd x^{\prime} + \int_{\Omega}e_{\rm xc}(\rho_{\varGamma_{n}}) \right)
\\
= \int_{\Omega} v_{\rm ext} \rho_{\varGamma} + \frac{1}{2}\int_{\Omega}\int_{\Omega}\frac{\rho_{\varGamma}(x)\rho_{\varGamma}(x^{\prime})}{|x-x^{\prime}|}\dd x\dd x^{\prime} + \int_{\Omega}e_{\rm xc}(\rho_{\varGamma}).
\qquad
\end{multline}
For the kinetic part in the energy \eqref{energy_gamma}, we obtain by using the Fatou's theorem for non-negative trace-class operators that
\begin{align}
\label{proof-kinetic-energy}
{\rm Tr}\left(|\nabla| \varGamma|\nabla|\right) 
& \leq \liminf _{n \rightarrow \infty} \operatorname{Tr}\left(|\nabla| \varGamma_{n}|\nabla|\right).
\end{align}
This together with \eqref{proof-ex} leads to the weak lower semicontinuity of the energy.

We shall finally consider the entropy \eqref{energy_gamma}.
By using the expression \eqref{gamma_dirac}
\begin{align*}
\varGamma = \sum_{i=1}^{\infty} f_i|\phi_i\rangle\langle\phi_i|, \qquad 
\varGamma_n = \sum_{i=1}^{\infty} f_{i,n}|\phi_{i,n}\rangle\langle\phi_{i,n}| ,
\end{align*}
and the fact that $\varGamma_n$ converge to $\varGamma$ weakly-* in $\mathfrak{S}^{1,1}\left(\Omega\right)$, we have that $\{f_{i,n}\}_n$ converges to $f_i$ for any $i\in\mathbb{N}_+$.
This together with the continuity of the function $x\ln x$ on $(0,1)$ leads to
\begin{eqnarray}
\label{proof-entropy}
\mathcal{S}(\varGamma) = 
\lim_{n\rightarrow\infty} \mathcal{S}(\varGamma_n)
\end{eqnarray}
Combining \eqref{proof-ex}, \eqref{proof-kinetic-energy} and \eqref{proof-entropy}, we can derive \eqref{lemma-2}, which completes the proof.
\end{proof}

\begin{remark}[Orbital-based formulation: the free energy]
\label{remark:orbital-0}
Under the traditional orbital-based formulation,
the Helmholtz free energy of an $N$-electron system is given by
\begin{eqnarray*}
F\big(\{\phi_i\}, \{f_i\}\big) = E\big(\{\phi_i\}, \{f_i\}\big) + S\big(\{f_i\}\big),
\end{eqnarray*}
where $\phi_i$ is the $i$th (single-particle) orbital, $f_i\in[0,1]$ is the occupation number of $\phi_i$, $E$ and $S$ are the energy and entropy respectively
\begin{align*}
& E\big(\{\phi_i\}, \{f_i\}\big) = \frac{1}{2}\sum_{i}f_i\int_{\mathbb{R}^3}|\nabla\phi_i|^2 + \int_{\mathbb{R}^3} v_{\rm ext}\rho + \frac{1}{2}\int_{\mathbb{R}^3}\int_{\mathbb{R}^3}\frac{\rho(x)\rho(x^{\prime})}{|x-x^{\prime}|}\dd x\dd x^{\prime} + \int_{\mathbb{R}^3} e_{\rm xc}(\rho),
\\[1ex]
& S\big(\{f_i\}\big) = \beta^{-1}\sum_i\Big( f_i\ln f_i + (1-f_i)\ln(1-f_i) \Big).
\end{align*}
Here the electron density is given by $\rho(x)=\sum_i f_i|\phi_i(x)|^2$.
\end{remark}

\section{Ground state solutions and optimality conditions}
\label{sec:gs}
\setcounter{equation}{0}

In finite temperature DFT, the ground states of the system can be obtained by minimizing the free energy with respect to the density matrices, that is
\begin{eqnarray}
\label{min_gamma}
E_0 := \min \Big\{ \F(\varGamma) :~ \varGamma\in \mathcal{K}_N \Big\}.
\end{eqnarray}
This variational problem will be used as the primary model for our numerical analysis throughout this paper.
Based on the coercivity and weak lower semicontinuity results derived in the previous section, we can establish the existence of minimizers for \eqref{min_gamma} (see for example the standard arguments in \cite{evans22}).
For our analysis, we shall further show that the minimizers of the problem satisfy some first- and second-order optimality conditions in the following.

\begin{remark}[Orbital-based formulation: variational form]
\label{remark:orbital-1}
With the orbital-based formulation of the free energy (see Remark \ref{remark:orbital-0}), we can find the equilibrium state of the system by solving the following constrained minimization problem
\begin{eqnarray}
\label{min_phi}
\min\left\{  F\big(\{\phi_i\}, \{f_i\}\big) :~\phi_i\in H_{\#}^1(\Omega), ~ 
\int_{\mathbb{R}^3}\phi_i\phi_j=\delta_{ij}, ~
\sum_i f_i=N \right\} .
\end{eqnarray}
This is equivalent to \eqref{min_gamma}, but within an orbital based formulation.
\end{remark}

\subsection{Euler-Lagrange equations}
\label{sec:MKS}

We first provide the first-order optimality condition for the solutions of \eqref{min_gamma}.
Let
\begin{equation*}
\mathcal{Y}=\mathfrak{S}^{1,1}\left(\Omega\right)\times\R 
\qquad {\rm and} \qquad \mathcal{Y}^*
=\left(\mathfrak{S}^{1,1}\left(\Omega\right)\right)^{*}\times\R.
\end{equation*}
The Lagrangian function associated with \eqref{min_gamma} can be written as
\begin{eqnarray}
\label{eq:Lagrangian}
\L\left(\varGamma, \mu \right) := \F\left(\varGamma\right)
- \mu\big({\rm Tr}\left(\varGamma\right)-N\big) 
\qquad{\rm for}~(\varGamma,\mu)\in \mathcal{Y} ,
\end{eqnarray}
where $\mu \in \mathbb{R}$ represents the Lagrange multiplier.
Note that we do not have to handle the constraint $0\leq \varGamma\leq 1$ due to the logarithmic function in the entropy $\mathcal{S}(\varGamma)$. 

We can then calculate its Fr\'echet derivative $\nabla_{\varGamma}\L: \mathcal{Y} \rightarrow \mathcal{Y}^*$, which is given by
\begin{eqnarray}
\label{eq:Gra}
\nabla_{\varGamma}\L\big( \varGamma,\mu \big) 
= \left[\begin{array}{cc}
\big(H\left(\rho_{\varGamma}\right) - \mu I\big) + 
\beta^{-1}\big(\ln\left(\varGamma\right)-\ln\left(I-\varGamma\right)\big)
\\[1ex]
{\rm Tr} \left(\varGamma\right) - N
\end{array}\right] 
\quad {\rm for}~(\varGamma, \mu)\in \mathcal{Y}.
\end{eqnarray}
If $\bar{\varGamma}$ is a minimizer of \eqref{min_gamma}, then there exists a Lagrangian multiplier $\bar{\mu}\in\R$, such that $(\bar{\varGamma},\bar{\mu})\in \mathcal{Y}$ satisfies the Euler-Lagrange equation
\begin{eqnarray}
\label{Euler-Lagrange}
\nabla_{\varGamma}\L\big( \bar{\varGamma},\bar{\mu} \big) 
= \left[\begin{array}{cc}
\bm{0}
\\[1ex]
0
\end{array}\right] \in \mathcal{Y}^*.
\end{eqnarray}
With a direct calculation, we see that \eqref{Euler-Lagrange} can be rewritten as
\begin{eqnarray}
\label{eq:G}
\G\big(\bar{\varGamma},\bar{\mu}\big) = \left[\begin{array}{cc}
\bm{0}
\\[1ex]
0
\end{array}\right],
\end{eqnarray}
with $\G:\mathcal{Y} \rightarrow \mathcal{Y}^*$ given by
\begin{align}
\label{def:G}
\G\big( \varGamma,\mu \big) := \left[\begin{array}{cc}
f_{\mu}\big(H(\rho_{\varGamma})\big) - \varGamma
\\[1ex]
{\rm Tr}(\varGamma) - N
\end{array}\right]
\quad {\rm for}~(\varGamma, \mu)\in \mathcal{Y}
\end{align}
and $f_{\mu}$ the Fermi-Dirac function
$f_{\mu}(x) = (1+e^{\beta(x-\mu)})^{-1}$.
The advantage of using \eqref{eq:G} instead of \eqref{Euler-Lagrange} is that the singular logarithm term $\ln(\varGamma)$ has been avoided in the expression of $\G$.

\begin{remark}[Orbital-based formulation: MKS equations]
\label{remark:orbital-2}
We can also derive the first-order optimality condition of the orbital-based optimization problem \eqref{min_phi}, which gives rise to the MKS equations \eqref{intro_MKS}.
We see that \eqref{eq:G} is equivalent to  \eqref{intro_MKS} but with a density matrix formulation. 
\end{remark}

\subsection{Second-order optimality condition}
\label{sec:second}

To establish the second-order optimality condition, we introduce a linear response operator $\chi : \mathfrak{S}^{1,1}(\Omega) \to \left(\mathfrak{S}^{1,1}\left(\Omega\right)\right)^{*}$, which represents the derivative $\nabla_{\varGamma}f_{\mu}\big(H(\rho_{\varGamma})\big)$ at the ground state $(\bar{\varGamma},\bar{\mu})$. 
In particular, for $\Psi\in \mathfrak{S}^{1,1}(\Omega)$, we define
\begin{align}
\label{eq:chi}
\nonumber
\chi\Psi :&= 
\nabla_{\varGamma}f_{\mu}\big(H(\rho_{\varGamma})\big) \Big|_{\varGamma=\bar{\varGamma},\mu=\bar{\mu}} \Psi 
\\[1ex]
& = \frac{1}{2\pi i} \oint_{\mathscr{C}} f_{\bar{\mu}}(z) \big(z-H(\rho_{\bar{\varGamma}})\big)^{-1}
\Big\langle \nabla_{\varGamma} v_{\rm eff}(\rho_{\varGamma})\big|_{\varGamma=\bar{\varGamma}}, \Psi \Big\rangle 
\big(z-H(\rho_{\bar{\varGamma}})\big)^{-1} \dd z,
\end{align}
where $f_{\bar{\mu}}$ is the Fermi-Dirac function defined as that in \eqref{def:G}, $\mathscr{C}$ is a smooth curve in the complex plane $\mathbb{C}$ enclosing the spectrum of $H(\rho_{\bar{\varGamma}})$ and avoiding all the singularities of $f_{\bar{\mu}}$.
Here we have exploited the contour representation in the spectral theory \cite{lax02}.

The linear response operator $\chi$ gives the response of the ground state density matrix to a small perturbation $\Psi$ on the density matrix \cite{cances102,gilliam09}.
We will require the following assumption in our analysis, which essentially indicates the stability of the linear response operator \cite{cances102,elu13}.

\vskip 0.2cm 

\noindent
{\bf A4}.
$I-\chi$ is positive definite and invertible. Moreover, there exists a constant $\kappa > 0$ such that $\|(I-\chi)^{-1}\|\leq \kappa$.

\vskip 0.2cm 

We remark that the assumption {\bf A4} is reasonable that can be satisfied in many practical models. 
For example, in the reduced Hartree-Fock model \cite{solovej91}, where the exchange-correlation term vanishes, it holds that 
\begin{align*}
\big\langle \chi\Psi, \Psi\big\rangle 
& = \frac{1}{2\pi i}\oint_\mathscr{C} {\rm Tr} 
\bigg( f_{\mu}(z) \big(z-H(\rho_{\bar{\varGamma}})\big)^{-1}
v_{\rm H}(\rho_{\Psi})\big(z-H(\rho_{\bar{\varGamma}})\big)^{-1} v_{\rm H}(\rho_{\Psi}) \bigg)
\\[1ex]
& = \sum_i\sum_j \frac{f_{\bar{\mu}}(\lambda_i)-f_{\bar{\mu}}(\lambda_j)}{\lambda_i - \lambda_j} 
\bigg| \big\langle\phi_i\big|v_{\rm H}(\rho_\Psi)\big|\phi_j\big\rangle \bigg|^2
\end{align*}
for any $\Psi \in S(L^2(\Omega))$.
Here the fraction gives the derivative of $f_{\bar{\mu}}(\cdot)$ when $\lambda_i=\lambda_j$ in the denominator.
This together with the fact that $f_{\bar{\mu}}$ is a decreasing function implies that $\chi$ is a non-positive operator and hence {\bf A4} could be true. 
We are then ready to give a second order optimality condition for the ground state solutions of \eqref{min_gamma}.

Note that the assumption {\bf A3} implies that $\G$ in \eqref{def:G} is Fr\`{e}chet differentiable.
We can calculate the Jacobian of $\G$ at $(\varGamma,\mu)$, denoted by $\J\big(\varGamma, \mu\big): \mathcal{Y} \rightarrow \mathcal{Y}^*$, with
\begin{eqnarray}
\label{eq:J}
\J\big(\varGamma, \mu\big)\big(\Psi, s\big)
= \left[\begin{array}{cc}
\big\langle\nabla_{\varGamma}f_{\mu}\big(H(\rho_{\varGamma})\big),\Psi\big\rangle - \Psi + s g_{\mu}\big(H(\rho_{\varGamma})\big)  
\\[1ex]
{\rm Tr}(\Psi)
\end{array}\right]
\quad {\rm for}~(\Psi, s)\in \mathcal{Y},
\end{eqnarray}
where $g_{\mu}(x) := \partial f_{\mu}(x)/\partial \mu = (-\beta e^{\beta(x-\mu)})(1+e^{\beta(x-\mu)})^{-2}$.
The following result shows that, at the ground state solution $\bar{\varGamma}$ of \eqref{min_gamma}, $\J\big(\bar{\varGamma}, \bar{\mu}\big)$ is an isomorphism.

\begin{lemma}
\label{stability}
Let $\bar{\varGamma}$ be a minimizer of \eqref{min_gamma} and $\bar{\mu}$ be the corresponding multiplier such that $(\bar{\varGamma},\bar{\mu})\in \mathcal{Y}$ satisfies \eqref{eq:G}.
If the assumptions {\bf A1}-{\bf A4} are satisfied, then $\J\big(\bar{\varGamma}, \bar{\mu}\big): \mathcal{Y}\rightarrow \mathcal{Y}^*$ is an isomorphism. 
Moreover, there is a constant $\sigma>0$ such that
\begin{eqnarray}
\label{ass:strong_stability}
\big\| \J\big(\bar{\varGamma},\bar{\mu}\big)^{-1} \big\| \leq \sigma.
\end{eqnarray}
\end{lemma}

\begin{proof}
We need to show that the equation
\begin{eqnarray}
\label{sec:op}
\J\big(\bar{\varGamma}, \bar{\mu}\big)(\Psi, s)=(\Phi, t).
\end{eqnarray}
is uniquely solvable for any $(\Phi, t)\in \mathcal{Y}^*$.
Using \eqref{eq:chi} and \eqref{eq:J}, we can rewrite \eqref{sec:op} as 
\begin{eqnarray}
\label{iso:J}
\begin{cases}
(\chi-I)\Psi + s g_{\mu}\big(H(\rho_{\varGamma})\big) = \Phi, 
\\[1ex]
{\rm Tr}(\Psi) = t.
\end{cases}
\end{eqnarray}
Then for any $(\Phi, t) \in \mathcal{Y}^*$, we have from the first equation of \eqref{iso:J} that
\begin{eqnarray}
\label{Psi_value}
\Psi = (\chi-I)^{-1} \big(\Phi - s g_{\mu}(H(\rho_{\varGamma}))\big).
\end{eqnarray}
Substituting \eqref{Psi_value} into the second equation of \eqref{iso:J} gives the expression for $s$
\begin{eqnarray}
\label{s_value}
s = \frac{{\rm Tr}\left(\left(\chi-I\right)^{-1}\Phi\right)
-t}{{\rm Tr}\left(\left(\chi-I\right)^{-1}g_{\mu}\left(H\left(\rho_{\varGamma}\right)\right)\right)}.
\end{eqnarray}

Note that $g_\mu(x) < 0$ implies that the eigenvalues of $g_{\mu}\big(H(\rho_{\varGamma})\big)$ are negative. 
Since $I-\chi$ is positive from the assumption {\bf A4}, we have that ${\rm Tr}\left(\left(\chi-I\right)^{-1}g_{\mu}\left(H\left(\rho_{\varGamma}\right)\right)\right)$ is negative, whose value depends on the inverse temperature $\beta$, the spectrum distribution of $H(\rho_{\varGamma})$, and the operator norm $\|(\chi - I)^{-1}\|$. 
Then we have 
\begin{eqnarray}
\label{s_estimate}
|s| \leq C_{\kappa} \big(\|\Phi\|_{\mathfrak{S}^{1,1}(\Omega)}+|t|\big),
\end{eqnarray}
where $C_\kappa$ is a positive constant depending on $\kappa$ in the assumption {\bf A4}. 

Finally, we substitute \eqref{s_value} into \eqref{Psi_value} and get
\begin{eqnarray}
\Psi = (\chi-I)^{-1} \bigg(\Phi - \frac{{\rm Tr}\left(\left(\chi-I\right)^{-1}\Phi\right)
-t}{{\rm Tr}\left(\left(\chi-I\right)^{-1}g_{\mu}\left(H\left(\rho_{\varGamma}\right)\right)\right)} g_{\mu}\left(H\left(\rho_{\varGamma}\right)\right)\bigg).
\end{eqnarray}
This together with the assumption {\bf A4} and the estimate \eqref{s_estimate} leads to
\begin{eqnarray*}
\|\Psi\|_{\mathfrak{S}^{1,1}(\Omega)} \leq  \tilde{C}_{\kappa}\big(\|\Phi\|_{\mathfrak{S}^{1,1}(\Omega)} + |t|\big)
\end{eqnarray*}
with some constant $\tilde{C}_{\kappa}$.
Therefore we get $\J\big(\bar{\varGamma},\bar{\mu}\big)$ is an isomorphism and \eqref{ass:strong_stability} holds.
\end{proof}

In the convergence analysis, we will also need the continuity of the first- and second-order derivatives, which is stated in the following lemma.

\begin{lemma}
\label{lemma:continuity}
Let $\bar{\varGamma}$ be a minimizer of \eqref{min_gamma} and $\bar{\mu}$ be the corresponding multiplier such that $\bar{y}=(\bar{\varGamma},\bar{\mu})\in \mathcal{Y}$ satisfies \eqref{eq:G}.
If the assumptions {\bf A1}-{\bf A3} are satisfied, then there exists $\delta>0$ such that both $\mathcal{G}$ and $\J$ are Lipschitz continuous in the neighborhood $B_{\delta}(\bar{y})$, more precisely, there exists constants $C_1$ and $C_2$ depending on $\bar{y}$, such that 
\begin{align}
\label{L-continuty-G}
& \left\|\mathcal{G}(y_1)
-\mathcal{G}(y_2)\right\|_{\mathcal{Y}^{*}}
\leq C_1\left\|y_1-y_2\right\|_{\mathfrak{S}^{1,1}(\Omega)},
\\[1ex]
\label{L-continuty-J}
& \left\|\J(y_1) - \J(y_2)\right\|_{\mathcal{L}(\mathcal{Y},\mathcal{Y}^{*})} \leq C_2 \left\|y_1-y_2\right\|_{\mathfrak{S}^{1,1}(\Omega)}
\end{align}
for any $y_1,y_2\in B_{\delta}(\bar{y})$.
\end{lemma}

\begin{proof}
We first prove the continuity of $\mathcal{G}$ in \eqref{L-continuty-G}. 
Note that continuity with respect to the chemical potential $\mu$  is quite obvious due to the smoothness of the Fermi-Dirac function with respect to $\mu$.
Then it is sufficient for us to show
\begin{eqnarray*}
\label{eq:lip_E}
\left\|f_{\mu}\big(H(\rho_{\varGamma_1})\big) - f_{\mu}\big(H(\rho_{\varGamma_2})\big)\right\|_{\left(\mathfrak{S}^{1,1}\left(\Omega\right)\right)^{*}} \leq C\|\varGamma_1 - \varGamma_2\|_{\mathfrak{S}^{1,1}(\Omega)} .
\end{eqnarray*}
Using the contour representation, we have 
\begin{align*}
& f_{\mu}\big(H(\rho_{\varGamma_1})\big) - f_{\mu}\big(H(\rho_{\varGamma_2})\big) 
\\[1ex] 
& = \frac{1}{2\pi i}\oint_\mathscr{C}f_\mu(z)\Big(\big(z-H(\rho_{\varGamma_1})\big)^{-1} - \big(z-H(\rho_{\varGamma_2})\big)^{-1} \Big) \dd z
\\[1ex] 
& = \frac{1}{2\pi i}\oint_\mathscr{C} 
 f_\mu(z)\big(z-H(\rho_{\varGamma_1})\big)^{-1}
 \big(H(\rho_{\varGamma_1})-
 H(\rho_{\varGamma_2})\big)\big(z-H(\rho_{\varGamma_2})\big)^{-1} \dd z
\end{align*}
with $\mathscr{C}$ being a smooth curve in the complex plane $\mathbb{C}$ enclosing the spectrum of $H(\rho_{\varGamma_1})$ and $H(\rho_{\varGamma_2})$. 
Since the the assumption {\bf A3} implies hold, we present
\begin{eqnarray*}
\label{LC-Hatree}
\left\|H(\rho_{\varGamma_1})-
H(\rho_{\varGamma_2})\right\|_{\left(\mathfrak{S}^{1,1}\left(\Omega\right)\right)^{*}} \leq C\|\varGamma_1 - \varGamma_2\|_{\mathfrak{S}^{1,1}(\Omega)},
\end{eqnarray*}
then we get \eqref{L-continuty-G}. 
    
For the continuity of $\mathcal{J}$ in \eqref{L-continuty-J}, it is sufficient to verify the Lipschitz continuity of
$\langle\nabla_{\varGamma}f_{\mu}\big(H(\rho_{\varGamma})\big),\Psi\rangle$ for any $\Psi\in \mathfrak{S}^{1,1}(\Omega)$. 
By using the contour representation in \eqref{eq:chi} and the assumption {\bf A3}, we can obtain by a similar calculation as above that
\begin{align*}
\big\langle\nabla_{\varGamma}f_{\mu}\big(H(\rho_{\varGamma_1})\big),\Psi\big\rangle - \big\langle\nabla_{\varGamma}f_{\mu}\big(H(\rho_{\varGamma_2})\big),\Psi\big\rangle 
\leq C\|\varGamma_1 - \varGamma_2\|_{\mathfrak{S}^{1,1}(\Omega)} \|\Psi\|_{\mathfrak{S}^{1,1}(\Omega)}.
\end{align*}
This completes the proof.
\end{proof}

\section{Numerical analysis of the Galerkin approximations}
\label{sec:numerical_anal}
\setcounter{equation}{0}

In this section, we shall investigate the finite dimensional approximations of \eqref{min_gamma}. 
We shall primarily justify the convergence of ground state approximations and derive an optimal {\it a priori} error estimate.

\subsection{Finite dimensional approximations}
\label{sec:discretization}

Let $V_n \subset H_{\#}^1(\Omega)~(n\in\mathbb{Z}_+)$ be a sequence of finite dimensional nested subspace satisfying 
\begin{align}
\label{Vn-dense}
V_n\subset V_m\quad\forall~n<m
\qquad{\rm and}\qquad
\lim_{n\rightarrow\infty}\sup_{\varphi\in H_{\#}^1(\Omega)}\inf_{\psi_n\in V_n} \|\varphi-\psi_n\|_{H^1} = 0.
\end{align}
This implies that $V_n$ become dense in $H_{\#}^1(\Omega)$ as $n\rightarrow\infty$.
Let $\pi_n: H_{\#}^1(\Omega) \to V_n$ be the $H^1$-projection such that 
\begin{align*}
\|\pi_n\varphi - \varphi\|_{H^1} = \inf_{\phi_n\in V_n}\|\phi_n-\varphi\|_{H^1}.
\end{align*}
We mention that in practical electronic structure calculations, the commonly employed discretization $V_n$ can take the form of plane-wave methods for periodic boundary conditions \cite{cances12, martin05} (see also the numerical experiments in Section \ref{sec:numerics}), or finite element methods for Dirichlet boundary conditions \cite{chen13, segall02, sur10}.

We can further define a corresponding projection $\Pi_n:\mathfrak{S}^{1,1}(\Omega)\to S\left(V_{n}\right)\cap\mathfrak{S}^{1,1}(\Omega)$ for the density matrix such that 
by
\begin{eqnarray}
\label{Pi_Gamma}
\left\|\Pi_{n} \varGamma - \varGamma\right\|_{\mathfrak{S}^{1,1}(\Omega)} = \inf_{\Psi_n\in S\left(V_{n}\right)\cap\mathfrak{S}^{1,1}(\Omega)} \left\|\Psi_n-\varGamma\right\|_{\mathfrak{S}^{1,1}(\Omega)} .
\end{eqnarray}
In particular, for $\displaystyle\varGamma = \sum_{i} f_i|\phi_{i}\rangle\langle\phi_{i}|$, we have $\displaystyle\Pi_n\varGamma := \sum_{i} f_i\big|\pi_n\phi_{i}\big\rangle\big\langle\pi_n\phi_{i}\big|$ and 
$\displaystyle\|\Pi_n\varGamma- \varGamma\|_{\mathfrak{S}^{1,1}(\Omega)} = \sum_i f_i \|\phi_i - \pi_n\phi_i\|_{H^1}^2$.
We have from \eqref{Vn-dense} that
\begin{eqnarray}
\label{pi_vargamma}
\lim _{n \rightarrow \infty}
\left\|\varGamma-\Pi_{n} 
\varGamma\right\|_{\mathfrak{S}^{1,1}(\Omega)} = 0
\qquad\forall~\varGamma\in \mathfrak{S}^{1,1}(\Omega).
\end{eqnarray}
Then we can define a finite dimensional subspace of $\mathcal{K}_N$ by
\begin{align*}
\mathcal{K}_{N,n}
:=\mathcal{K}_N \cap \mathcal{S}(V_n)
= \Big\{\varGamma_{n} \in \mathcal{S}(V_{n})\cap\mathfrak{S}^{1,1}(\Omega)
~:~ 
0\leq\varGamma_{n}\leq1, ~ \operatorname{Tr}\left(\varGamma_{n}\right) = N \Big\}.
\end{align*}
With the above discretizations,  the problem \eqref{min_gamma} can be approximated by 
\begin{equation}
\label{dis:gal-mod}
E_n := \min \Big\{\F\left(\varGamma_{n}\right) :~ \varGamma_{n} \in \mathcal{K}_{N,n}\Big\},
\end{equation}
where the free energy functional $\F$ is given by \eqref{energy_gamma}. 

We can then derive the first- and second-order optimality condition of \eqref{dis:gal-mod}, which can be viewed as the finite dimensional approximations of the continuous counterparts.
Let 
\begin{equation*}
\mathcal{Y}_n=\big(\mathcal{S}(V_n)\cap \mathfrak{S}^{1,1}(\Omega)\big)\times\R\qquad {\rm and} \qquad \mathcal{Y}_n^*=\big(\mathcal{S}(V_n)\cap \mathfrak{S}^{1,1}(\Omega)\big)^*\times\R.
\end{equation*}
By using the same arguments as those from \eqref{eq:Gra} to \eqref{eq:G}, we can derive the first-order optimality condition of \eqref{dis:gal-mod}:
If $\bar{\varGamma}_n$ is the minimizer of \eqref{dis:gal-mod}, then there exists a Lagrange multipiler $\bar{\mu}_n$ such that
\begin{align*}
\label{eq:gal-G=0}
\G_n\big( \bar{\varGamma}_n,\bar{\mu}_n \big) 
=\left[\begin{array}{cc}
\bm{0}
\\[1ex]
0
\end{array}\right] \in \mathcal{Y}_n^*,
\end{align*}
where $\G_n:\mathcal{Y}_n \rightarrow 
\mathcal{Y}_n^*$ is given by
\begin{eqnarray}
\label{eq:gal-G}
\G_n\big( \varGamma_n,\mu_n \big)  := \left[\begin{array}{cc}
f_{\mu_n}\big(H(\rho_{\varGamma_n})\big) -  \varGamma_n
\\[1ex]
{\rm Tr}(\varGamma_n) - N
\end{array}\right]
\qquad {\rm for}~(\varGamma_n, \mu_n)\in \mathcal{Y}_n.
\end{eqnarray}

\begin{remark}[Orbital-based formulation: the discretized MKS equations]
\label{remark:orbital-3}
In practical calculations, it is common to solve the following discretized MKS equation (instead of the operator form \eqref{eq:gal-G}) in order to obtain the ground state approximation:
Find $\left(\psi_{i,n}, \lambda_{i,n}\right) 
\in V_n \times \mathbb{R}~(i=1,2,\cdots)$ and $\mu_n\in\R$, such that $(\psi_{i,n},\psi_{j,n})=\delta_{ij}$, $\sum_i f_{\mu_n}(\lambda_{i,n})=N$ and
\begin{equation}
\label{orbital:phi:n}
a\left(\rho_n;\psi_{i,n}, \phi_n\right)=\lambda_{i,n}\left(\psi_{i,n}, \phi_n\right) 
\qquad \forall ~\phi_n \in V_n ,
\end{equation}
where $\rho_n := \sum_i f_{\mu_n}(\lambda_{i,n})|\psi_{i,n}|^2$ and
$a(\rho; \psi, \phi) := \frac{1}{2} \int \nabla \psi \cdot \nabla \phi+\int v_{\rm eff}(\rho)\psi \phi$.
\end{remark}

\vskip 0.2cm

We can further give the second-order optimality condition of \eqref{dis:gal-mod}.
We first write down the derivative of $\mathcal{G}_n$ by $\J_n\big(\varGamma_n, \mu_n\big):\mathcal{Y}_n \rightarrow 
\mathcal{Y}_n^*$, with
\begin{eqnarray*}
\label{eq:gal-J}
\J_n\big(\varGamma_n, \mu_n\big)
\big(\Psi, s\big)
=\left[\begin{array}{cc}
\big\langle\nabla_{\varGamma_n}f_{\mu_n}\big(H(\rho_{\varGamma_n})\big),\Psi\big\rangle - \Psi + s g_{\mu_n}\big(H(\rho_{\varGamma_n})\big)
\\[1ex]
{\rm Tr}(\Psi)
\end{array}\right]
\quad {\rm for}~(\Psi,s)\in \mathcal{Y}_n
\end{eqnarray*}
with $g_{\mu_n}(x) = \partial f_{\mu_n}(x)/\partial \mu_n = (-\beta e^{\beta(x-\mu_n)})(1+e^{\beta(x-\mu_n)})^{-2}$. 
Then the following result states the second-order optimality condition, which can be viewed as a discrete counterpart of Lemma \ref{stability}.
It indicates that $\J_n$ is also an isomorphism around $(\bar{\varGamma},\bar{\mu})$ when $n$ is sufficiently large.

\begin{lemma}
\label{stabilityvn}
Let $(\bar{\varGamma},\bar{\mu})\in \mathcal{Y}$ be a solution of \eqref{eq:G}.
If the assumptions {\bf A1}-{\bf A4} are satisfied, then there exists $n_0\in\mathbb{N}_+$ such that for any $n>n_0$, $\J_n\big(\Pi_n\bar{\varGamma},\bar{\mu}\big) : \mathcal{Y}_n \to \mathcal{Y}_n^*$ is an isomorphism and
\begin{eqnarray}
\label{Jn:iso}
\big\|\mathcal{J}_n\big(\Pi_n\bar{\varGamma},\bar{\mu}\big)^{-1}\big\| \leq 2\sigma
\end{eqnarray}
with $\sigma$ the stability constant given in \eqref{ass:strong_stability} of Lemma \ref{stability}.
\end{lemma}

\begin{proof}
We have from Lemma \ref{stability} that $\mathcal{J}(\bar{\varGamma},\bar{\mu})$ is an isomorphism and the following inf-sup condition holds 
\begin{eqnarray*}
\inf_{(\Psi,s)\in \mathcal{Y}} \sup_{(\Phi,t)\in \mathcal{Y}} \frac{\big\langle\mathcal{J}(\bar{\varGamma},\bar{\mu}) (\Psi,s), (\Phi,t) \big\rangle}{\|(\Psi,s)\| \cdot\|(\Phi,t)\|} \geq \sigma^{-1}.
\end{eqnarray*}
Then this together with the fact $\mathcal{Y}_n\subset\mathcal{Y}$ implies
\begin{eqnarray*}
\sup_{(\Phi,t)\in \mathcal{Y}} \frac{\big\langle\mathcal{J}(\bar{\varGamma},\bar{\mu}) (\Psi_n,s), (\Phi,t) \big\rangle}{\|(\Psi_n,s)\| \cdot\|(\Phi,t)\|} \geq \sigma^{-1} 
\qquad \forall~(\Psi_n,s) \in \mathcal{Y}_n,
\end{eqnarray*}
Let $\Pi_n$ be the projection defined by \eqref{Pi_Gamma}.
For any $(\Phi,t)\in\mathcal{Y}$, we have
\begin{align*}
\big\langle\mathcal{J}(\bar{\varGamma},\bar{\mu})(\Psi_n,s), (\Pi_n\Phi,t)\big\rangle
&= \big\langle\mathcal{J}(\bar{\varGamma},\bar{\mu})(\Psi_n,s), (\Phi,t)\big\rangle - \big\langle\mathcal{J}(\bar{\varGamma},\bar{\mu})(\Psi_n,s), (\Phi - \Pi_n\Phi,t)\big\rangle .
\end{align*}
Then we have from the convergence \eqref{pi_vargamma} that there exists $n_1>0$, such that 
\begin{align}
\label{proof:stability:n}
\inf_{(\Psi_n,s)\in \mathcal{Y}_n} \sup_{(\Phi_n,t)\in \mathcal{Y}_n} \frac{\big\langle\mathcal{J}(\bar{\varGamma},\bar{\mu}) (\Psi_n,s), (\Phi_n,t) \big\rangle}{\|(\Psi_n,s)\| \cdot\|(\Phi_n,t)\|} \geq \frac{3}{4}\sigma^{-1} 
\end{align}
for any $n \geq n_1$.
Then we have from
\begin{align*}
& \big\langle\mathcal{J}_n(\Pi_n\bar{\varGamma},\bar{\mu})(\Psi_n,s),(\Phi_n,t)\big\rangle 
\\[1ex]
= & \big\langle\mathcal{J}\left(\bar{\varGamma},\bar{\mu}\right)(\Psi_n,s),(\Phi_n,t)\big\rangle
-\big\langle \big(\mathcal{J}(\bar{\varGamma},\bar{\mu}) -\mathcal{J}(\Pi_n\bar{\varGamma},\bar{\mu})\big) (\Psi_n,s),(\Phi_n,t)\big\rangle
\end{align*}
and the Lipschtiz continuity of $\mathcal{J}$ in Lemma \ref{lemma:continuity} that there exists $n_2>0$, such that
\begin{align}
\label{proof:stability:n2}
\inf_{(\Psi_n,s)\in \mathcal{Y}_n} \sup_{(\Phi_n,t)\in \mathcal{Y}_n} \frac{\big\langle\mathcal{J}_n (\Pi_n\bar{\varGamma},\bar{\mu}) (\Psi_n,s), (\Phi_n,t) \big\rangle}{\|(\Psi_n,s)\| \cdot\|(\Phi_n,t)\|} \geq \frac{1}{2}\sigma^{-1}
\end{align}
for any $n \geq n_2$.
Let $n_0=\max\big\{n_1,n_2\big\}$, we have from \eqref{proof:stability:n2} that $\J_n\big(\Pi_n\bar{\varGamma},\bar{\mu}\big)$ is an isomorphism and \eqref{Jn:iso} holds.
This completes the proof.
\end{proof}

\subsection{Convergence analysis}
\label{sec:convergence}

We can now derive the following theorem that ensures the convergence of the finite dimensional approximations, say the solutions of \eqref{dis:gal-mod} can converge to the ground state of the continuous problem \eqref{min_gamma}.

\begin{theorem}
\label{theo_converge}
Let $\bar{\varGamma}$ be the ground state solution of \eqref{min_gamma} $E_0$ be the ground state energy. 
Let $\bar{\varGamma}_{n}\in \mathcal{K}_{N,n}$ be the finite dimensional approximations obtained by \eqref{dis:gal-mod} for $n\in \mathbb{N}_+$.
If the assumptions {\bf A1} and {\bf A2} are satisfied, then
\begin{equation}
\label{eq:convergence}
\lim _{n \rightarrow \infty} \F\left(\bar{\varGamma}_{n}\right) = E_0.
\end{equation}
Moreover, if $\bar{\varGamma}$ is a unique minimizer of \eqref{min_gamma}, then $\rho_{\bar{\varGamma}_{n}}$ converges to  $\rho_{\bar{\varGamma}}$ strongly.
\end{theorem} 

\begin{proof}
Note that Lemma \ref{lemma:coercivity} implies that $\left\{\bar{\varGamma}_{n}\right\}_{n=1}^{\infty}$ is bounded in $\mathcal{K}_{N,n}$. 
Then by applying the Banach–Alaoglu Theorem \cite{reed80}, there exists a subsequence 
$\left\{\bar{\varGamma}_{n_{j}}
\right\}_{j=1}^{\infty}$, such that $\bar{\varGamma}_{n_{j}}$ converge to some $\bar{\varGamma}_{\infty} \in \mathfrak{S}^{1,1}(\Omega)$ in the weak-* topology of $\mathfrak{S}^{1,1}(\Omega)$. 
We shall first show that
\begin{eqnarray}
\label{convergence of energy}
\F(\bar{\varGamma}_{\infty}) 
= E_0.
\end{eqnarray}
By using $\bar{\varGamma}_{n_j} \stackrel{\mathrm{w}^*}{\rightharpoonup} \bar{\varGamma}_{\infty}$ and Lemma \ref{lemma:semicontinuity}, we have
\begin{eqnarray*}
\liminf _{j \rightarrow \infty} \F\left(\bar{\varGamma}_{n_j}\right) \geq \F\left(\bar{\varGamma}_{\infty}\right).
\end{eqnarray*}
Note that $\Pi_{n_j}\bar{\varGamma}\in \mathcal{K}_{N,n_j}$ and the projection satisfies \eqref{pi_vargamma}, we have
\begin{eqnarray*}
\F\left(\bar{\varGamma}_{\infty}\right) \leq \liminf_{j \rightarrow \infty} \F\left(\bar{\varGamma}_{n_j}\right) 
\leq \liminf_{j \rightarrow \infty} 
\F\left(\Pi_{n_j} \bar{\varGamma}\right) 
= \F\left(\bar{\varGamma}\right) = E_0.
\end{eqnarray*}
This together with the fact $\bar{\varGamma}$ is the solution of \eqref{min_gamma} indicates \eqref{convergence of energy}.

If $\bar{\varGamma}$ is a unique solution of \eqref{min_gamma}, we have $\bar{\varGamma}=\bar{\varGamma}_{\infty}$, and hence $\bar{\varGamma}_{n} \stackrel{\mathrm{w}^*}{\rightharpoonup} \bar{\varGamma}$. 
By using the same argument as that in the proof of Lemma \ref{lemma:semicontinuity}, we can derive the convergence of the density, that is, $\rho_{\bar{\varGamma}_{n}}\rightarrow\rho_{\bar{\varGamma}}$ strongly in $L^{2}(\Omega)$.
\end{proof}

\begin{remark}[the nonconvex case]
\label{remark:nonconvex}
Note that the condition on the uniqueness of ground state solution $\bar{\varGamma}$ in Theorem \ref{theo_converge} may not always hold.
In that case, we can still find a subsequence with energy limit converging to $E_0$ for each $\bar{\varGamma}$. In detail, we can define the set of the ground state solutions of \eqref{min_gamma} by
\begin{equation*}
\Theta := \{\bar{\varGamma} \in \mathcal{K}_{N}: 
\bar{\varGamma}~{\rm solves}~\eqref{min_gamma} \}\qquad {\rm and} \qquad \Theta_n := \{\bar{\varGamma}_n \in \mathcal{K}_{N,n}: \bar{\varGamma}_n~{\rm solves}~\eqref{dis:gal-mod} \}.
\end{equation*}
With similar arguments as in the proof of Theorem \ref{theo_converge}, we can prove that each subsequence $\left\{\bar{\varGamma}_{n_{j}}
\right\}_{j=1}^{\infty}$ in $\Theta_n$ converges to a certain $\bar{\varGamma}_n$ in $\Theta$ and moreover 
\begin{equation*}
\lim _{n \rightarrow \infty} \mathcal{D}\left(\Theta_n, \Theta\right):= \sup_{\bar{\varGamma}\in\Theta} \inf_{\bar{\varGamma}_n\in\Theta_n} \|\rho_{\bar{\varGamma}_{n}} - \rho_{\bar{\varGamma}}\|_{L^2}=0.
\end{equation*}
\end{remark}

\subsection{A priori error estimate}
\label{sec:priori}

In addition to the convergence result, we can further derive an {\it a priori} error estimate for the finite dimensional approximations of the ground state solution. 
In particular, we will utilize the inverse function theorem \cite{ortner09} to establish a optimal convergence rate for the ground state approximations.

\begin{theorem}
\label{error-estimate}
Let $\bar{\varGamma}\in \mathcal{K}_N$ be a solution of  \eqref{min_gamma}.
If the assmuptions {\bf A1-A4} are satisfied, then there exist $\delta>0$ and $n_0\in\mathbb{N}_+$ 
such that for any $n>n_0$, \eqref{dis:gal-mod}  has a unique solution $\bar{\varGamma}_n\in \mathcal{K}_{N,n}\cap B_\delta(\bar{\varGamma})$. 
Moreover, there exists a constant $C$ such that
\begin{eqnarray*}
\label{eq:apriori}
\|\bar{\varGamma}_n -\bar{\varGamma}\|_{\mathfrak{S}^{1,1}(\Omega)}  
\leq C\inf_{\varGamma_n\in S(V_n)\cap \mathfrak{S}^{1,1}(\Omega)} \|\varGamma_n- \bar{\varGamma}\|_{\mathfrak{S}^{1,1}(\Omega)}.
\end{eqnarray*}
\end{theorem}

\begin{proof}
{\it (1) Consistency.}
Utilizing the fact that $\left.\mathcal{G}
(\bar{\varGamma},\bar{\mu})\right|_{\mathcal{K}_{N,n}}=0$
and $\mathcal{G}_{n}\left(\Pi_{n} \bar{\varGamma},\bar{\mu}\right)
=\left(\mathcal{G}\left(\Pi_{n}\bar{\varGamma},\bar{\mu}\right)
\right|_{\mathcal{K}_{N,n}}$, we proceed as follows:
\begin{align*}
\left\|\mathcal{G}_{n}\left(\Pi_{n} \bar{\varGamma},\bar{\mu}\right)
\right\|_{(\mathfrak{S}^{{1,1}}(\Omega))^{*}}
& = \left\|\left.\mathcal{G}\left(\Pi_{n}\bar{\varGamma},
\bar{\mu}\right)\right|_{\mathcal{K}_{N, n}}
-\left.\mathcal{G}(\bar{\varGamma},\bar{\mu})\right|
_{\mathcal{K}_{N, n}}
\right\|_{(\mathfrak{S}^{1,1}(\Omega))^{*}}
\\[1ex]
& = \left\|\mathcal{G}\left(\Pi_{n} \bar{\varGamma},
\bar{\mu}\right)-\mathcal{G}(\bar{\varGamma},\bar{\mu})\right\|
_{(\mathfrak{S}^{1,1}(\Omega))^{*}}.
\end{align*}
The Lipschitz continuity of $\mathcal{G}$ yields
\begin{eqnarray*}
\left\|\mathcal{G}_{n}\left(\Pi_{n} \bar{\varGamma},\bar{\mu}\right)
\right\|_{(\mathfrak{S}^{{1,1}}(\Omega))^{*}} 
\leq C\left\|\bar{\varGamma}-\Pi_{n} 
\bar{\varGamma}\right\|_{\mathfrak{S}^{1,1}(\Omega)}. 
\end{eqnarray*}
Hence, we have completed the proof of stability.

{\it (2) Stability.}
Note that Lemma \ref{stabilityvn} ensures that there exists $n_0\in\mathbb{N}_+$, such that $\|\mathcal{J}_n (\Pi_n\bar{\varGamma},
\bar{\mu})^{-1}\|_{(\mathfrak{S}^{1,1}(\Omega))^{*}} \leq 2\sigma$ for any $n>n_0$.

{\it (3) Apply the Inverse Function theorem.}
According to the Inverse Function theorem \cite{ortner11}, we have that for $n>n_0$, there exists some $\delta >0$ and a unique solution $\bar{\varGamma}_n$ in $\mathcal{K}_{N,n}\cap B_\delta(\bar{\varGamma})$, such that $\bar{\varGamma}_n$ solves \eqref{min_gamma} and
\begin{eqnarray}
\label{ift}
\|\bar{\varGamma}_n - \Pi_n\bar{\varGamma}\| _{\mathfrak{S}^{1,1}(\Omega)} 
\leq C_{\sigma}\|\bar{\varGamma} - \Pi_n\bar{\varGamma}\|_{\mathfrak{S}^{1,1}(\Omega)}
\end{eqnarray}
with positive constant $C_{\sigma}$.
Finally, we can derive from \eqref{ift} that
\begin{eqnarray*}
\left\|\bar{\varGamma}-\bar{\varGamma_n}\right\|_{\mathfrak{S}^{1,1}(\Omega)} 
\leq \left\|\bar{\varGamma}-\Pi_n\bar{\varGamma}\right\|_{\mathfrak{S}^{1,1}(\Omega)}
+ \left\|\Pi_n \bar{\varGamma}-\bar{\varGamma_n}\right\|_{\mathfrak{S}^{1,1}(\Omega)}
\leq C\|\bar{\varGamma} - \Pi_n\bar{\varGamma}\|_{\mathfrak{S}^{1,1}(\Omega)},
\end{eqnarray*}
which concludes the proof.
\end{proof}

\begin{remark}
[Orbital-based formulation: error estimate]
From the error estimates for ground state density matrix in Theorem \ref{error-estimate}, we can also demonstrate the convergence of the numerical approximation of the orbitals:
\begin{eqnarray}
\label{orbital_error}
\|\phi_i-\phi_{i,n}\|_{H_0^1(\Omega)} 
\leq 
C\|\bar{\varGamma}-\bar{\varGamma}_n\|_{\mathfrak{S}^{1,1}(\Omega)} 
\leq C \max_{1\leq i\leq N} \|\pi_n\phi_i-\phi_i\|_{H_0^1(\Omega)} ,
\end{eqnarray}
where $\phi_{i,n}~(i=1,2,\cdots)$ are the solutions of \eqref{orbital:phi:n}.
The decay of the numerical errors are then determined by the discretization methods, for example, it could be an exponential decay for plane-wave method (see Section \ref{sec:numerics}) and algebraic decay for finite element methods.
\end{remark}

\section{Numerical experiments}
\label{sec:numerics}
\setcounter{equation}{0}

In this section, we present some numerical experiments to support the theory of this work.
In particular, we consider two typical material systems: 
(1) silicon with the diamond structure and (2) aluminum with the FCC lattice (see Figure \ref{fig:Si} and \ref{fig:Al} for the crystal structure).

All simulations  were performed on a workstation equipped with an Intel Xeon W-3275M 56-core processor and 1.5 TB of RAM, by using the Julia \cite{Julia} package DFTK.jl \cite{DFTKjcon}.
We use the PW LDA exchange-correlation energy functional \cite{perdew92} and the HGH pseudopotential \cite{hamann79}. 
The computational precision is set to a tolerance of $10^{-6}$.
All numerical results are given in atomic unit (a.u.).
To evaluate the numerical errors, we consider numerical solutions obtained on very fine discretizations as the exact solutions.

The discretization scheme we used in the numerical simulation is the plane-wave method \cite{lin19,martin05}, which is the most widely used method for electronic structure calculations of materials. 
Let $\LL$ be the periodic lattice, $\Omega$ be the unit cell with respect to the lattice, and $\LL^*$ be the dual lattice of $\LL$.
Then the plane-wave basis function can be written as
\begin{align*}
e_{\Gp}({\rr})=|\Omega|^{-1/2} e^{i\Gp\cdot\rr}
\qquad \Gp\in\LL^* .
\end{align*} 
The family $\{e_{\Gp}\}_{\Gp\in\LL^*}$ forms an orthonormal basis set of 
\begin{eqnarray*}
L_{\#}^2(\Omega)=\{u\in L^2_{\rm loc}(\mathbb{R}^{d})~:~u~{\rm is}~\LL{\rm -periodic}\}. 
\end{eqnarray*}
For $u\in L_{\#}^2(\Omega)$, we have
\begin{equation*}
u({\rr})=\sum_{\Gp\in\LL^*}\hat{u}_{\Gp} e_{\Gp}(\rr)
\qquad{\rm with}\qquad 
\hat{u}_{\Gp} = (e_{\Gp},u)_{L^2_{\#}(\Omega)} = |\Omega|^{-1/2}\int_{\Omega} u(\rr)e^{-i\Gp\cdot\rr}\dd{\rr}.
\end{equation*}
In practical simulations, we will use an energy cut-off $\Ec>0$ and approximate the eigenfunctions in the following finite dimensional space 
\begin{equation}
\label{def:spaceE}
V_{\Ec}(\Omega) := \bigg\{ u\in L^2_{\#}(\Omega)~:~u(\rr) = \sum_{\Gp\in\LL^*,~|\Gp|^2\leq 2\Ec} c_{\Gp}e_{\Gp}({\rr})
\bigg\}.
\end{equation}
For $u\in L^2_{\#}(\Omega)$, the best approximation of $u$ in $V_{\Ec}(\Omega)$ is given by the following projection
\begin{align*}
\pi_{\Ec} u = \sum_{\Gp\in\LL^*,~|\Gp|^2\leq 2\Ec} \hat{u}_{\Gp} e_{\Gp}(\rr).
\end{align*} 
The more regularity $u$ has, the faster this projection converge to $u$.
Note that we are mainly concerned with the $H^1$-norm error in this work. 
We have from \cite{Claudio} that for $u\in H^s_{\#}(\Omega)$ the convergence is super-algebraic
\begin{equation}
\label{error:interpolation:planewave}
\|u-\pi_{\Ec} u\|_{H^1_{\#}(\Omega)}  \leq C_{u,s} \Ec^{(1-s)}
\qquad\forall~s>1 ,
\end{equation}
and from \cite{Trefethen} that for $u$ analytic the convergence is exponential
\begin{equation}
\label{error:interpolation:planewave:exp}
\|u-\pi_{\Ec} u\|_{H^1_{\#}(\Omega)}  \leq C_u e^{-\alpha_u\Ec} 
\end{equation}
with some positive constants $C_{u,s}$, $C_u$ and $\alpha_u$.
Since we are using smooth pseudopotentials in our simulations, we can assume that the eigenfunctions are sufficiently smooth.

\vskip 0.1cm

{\bf Example 1.}
Consider a silicon (Si) system with the diamond structure (see Figure \ref{fig:Si} for the crystal structure). 
We choose a system with $4\times 2\times 2$ unit cells, which contains 32 Si atoms.
The numerical errors with respect to the energy cutoffs are shown in Figure \ref{fig:Si}.
We observe that the ground state energy and electron density both converge exponentially fast with respect to the energy cutoffs, at different temperatures.
This matches perfectly with our theoretical perspective.

\begin{figure}[htb!]
\vskip 0.1cm
\centering
\includegraphics[width=4.2cm]{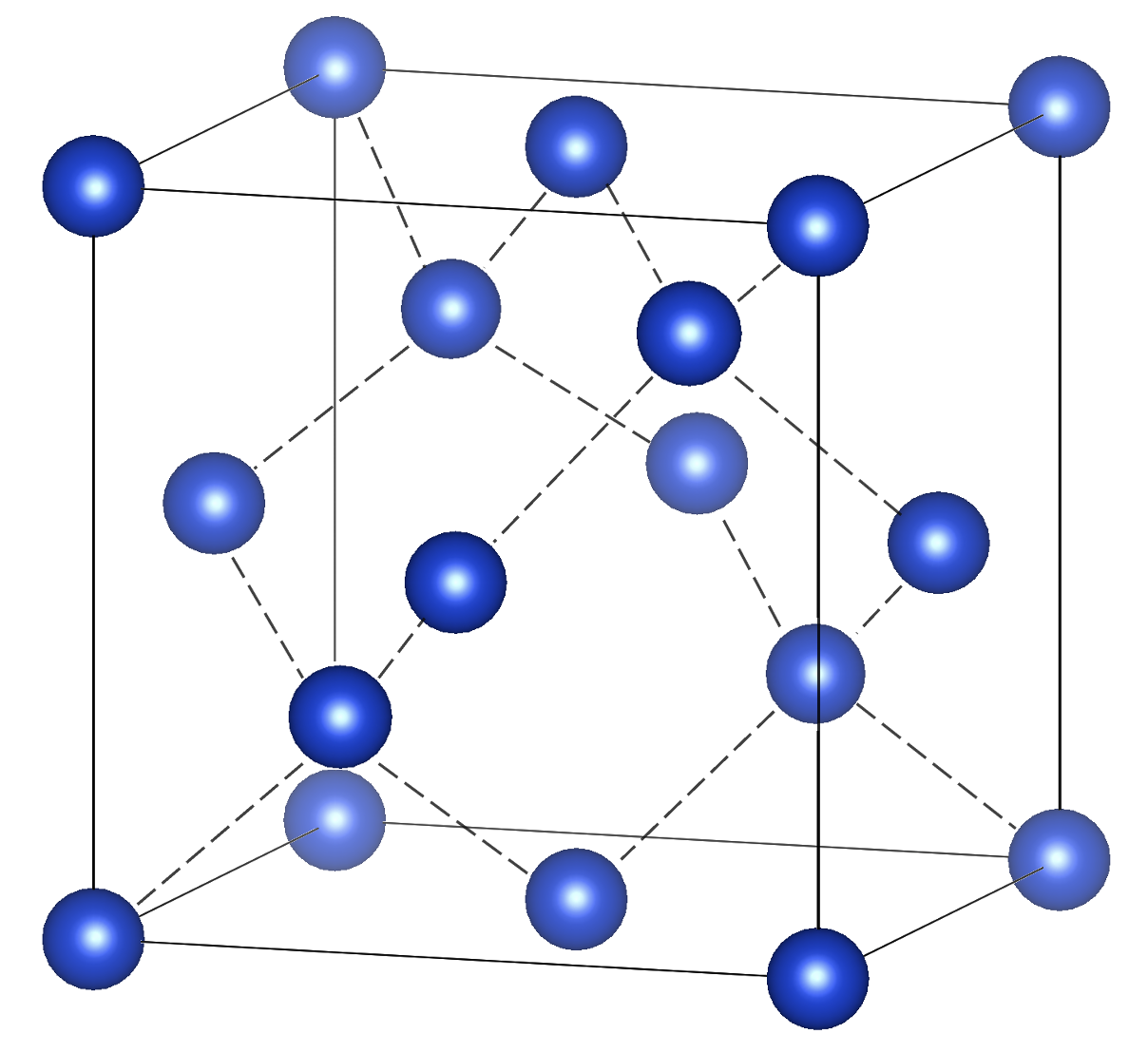}
\hskip 0.3cm
\includegraphics[width=5.0cm]{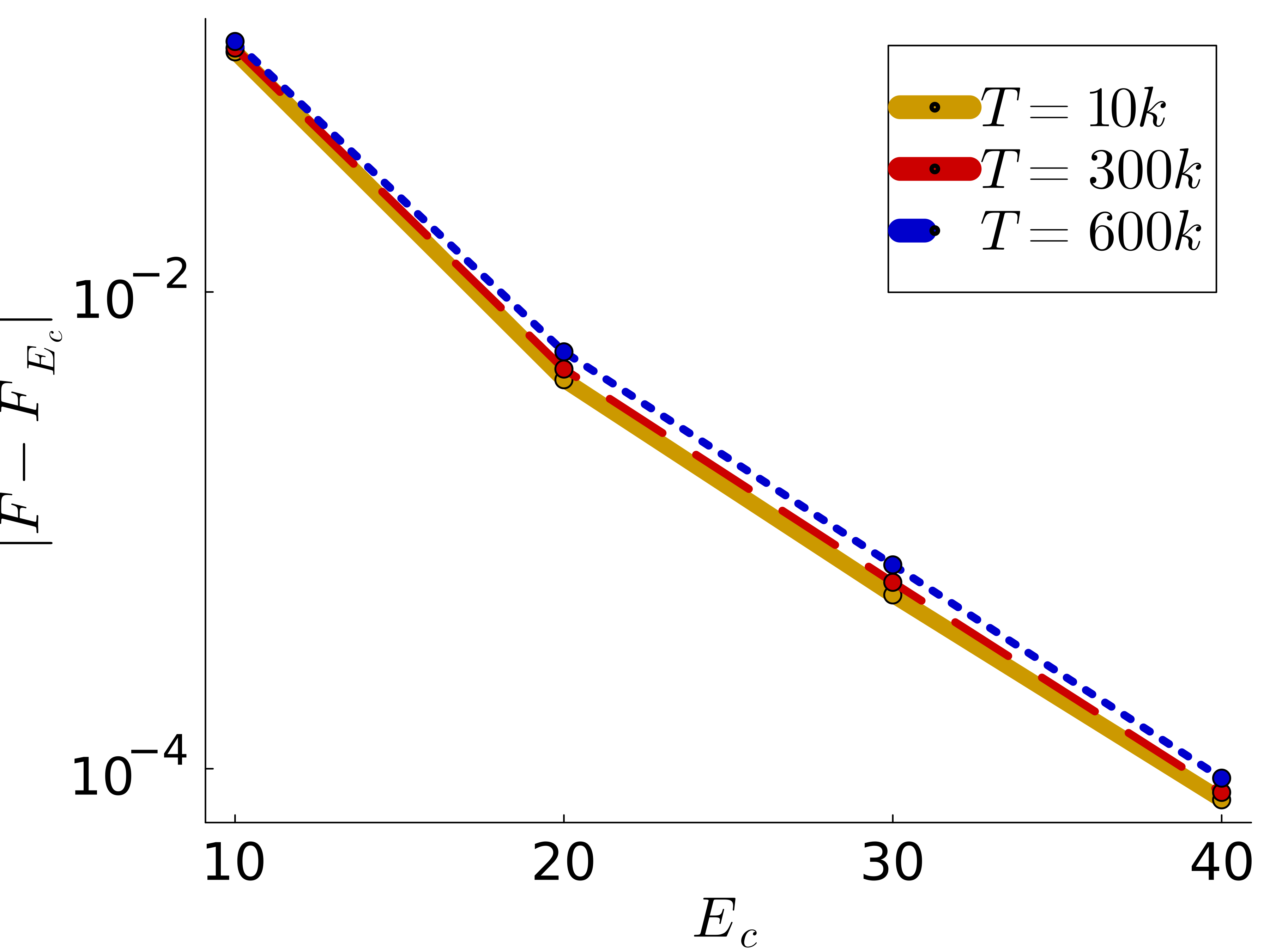}
\hskip 0.3cm
\includegraphics[width=5.0cm]{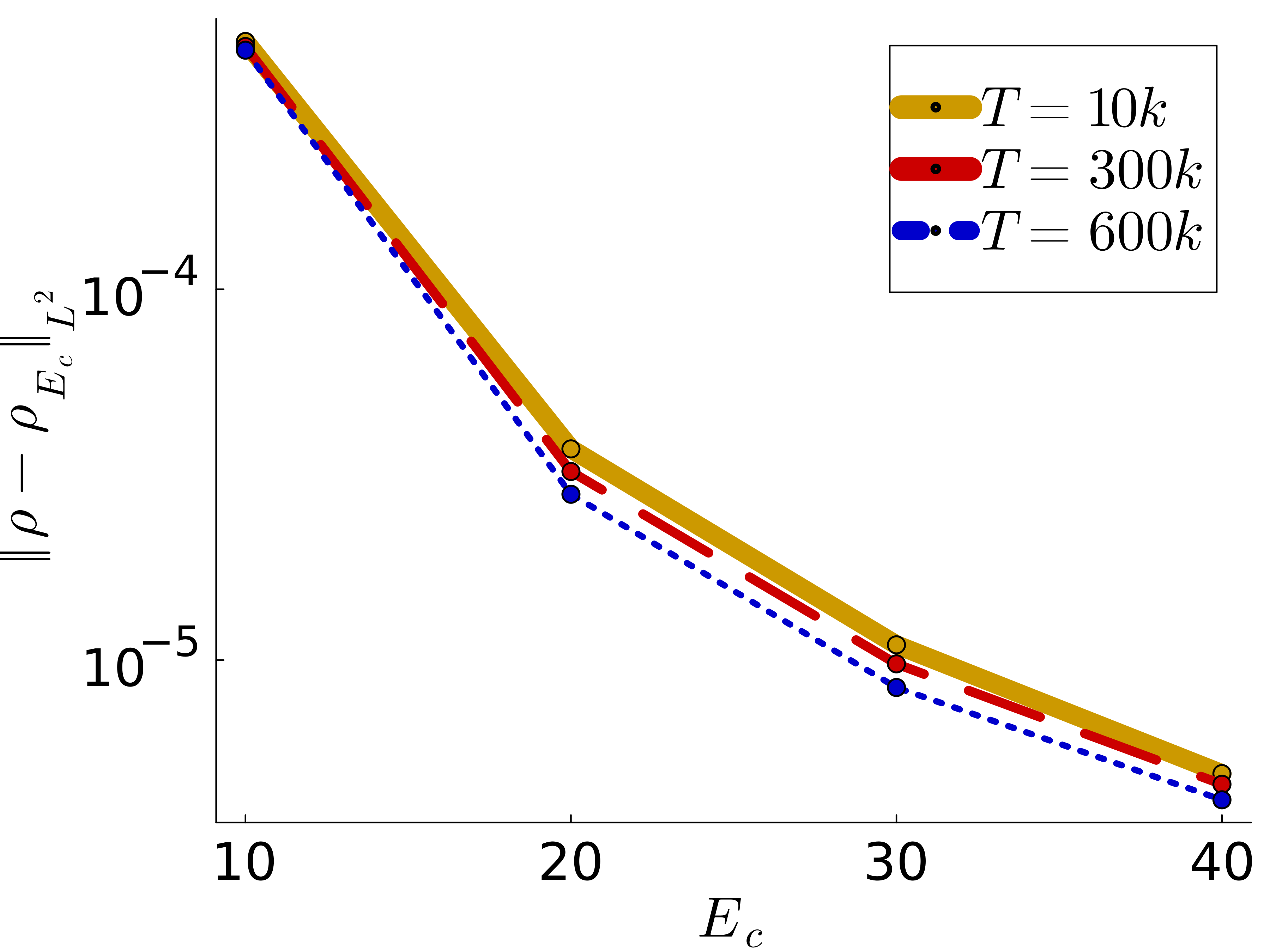}
\caption{(Example 1. Silicon) 
Left: diamond structure in a unit cell. Middle: energy error with respect to the energy cutoffs. Right: $L^2$-error of the electron density with respect to the energy cutoffs.}
\label{fig:Si}
\vskip 0.2cm
\end{figure}

\vskip 0.1cm

{\bf Example 2.}
Consider an aluminium system with the face-centered cubic (FCC) structure, see Figure \ref{fig:Al}.
We choose a system with $4\times 2\times 2$ unit cells, which contains 64 Al atoms.
We show in Figure \ref{fig:Al} the decay of numerical errors of the ground state energy and electron density.
Both errors exhibit an exponential decay with respect to the plane-wave energy cutoffs.
This observation is again consistent with our theory.

\begin{figure}[htb!]
\vskip 0.1cm
\centering
\includegraphics[width=4.2cm]{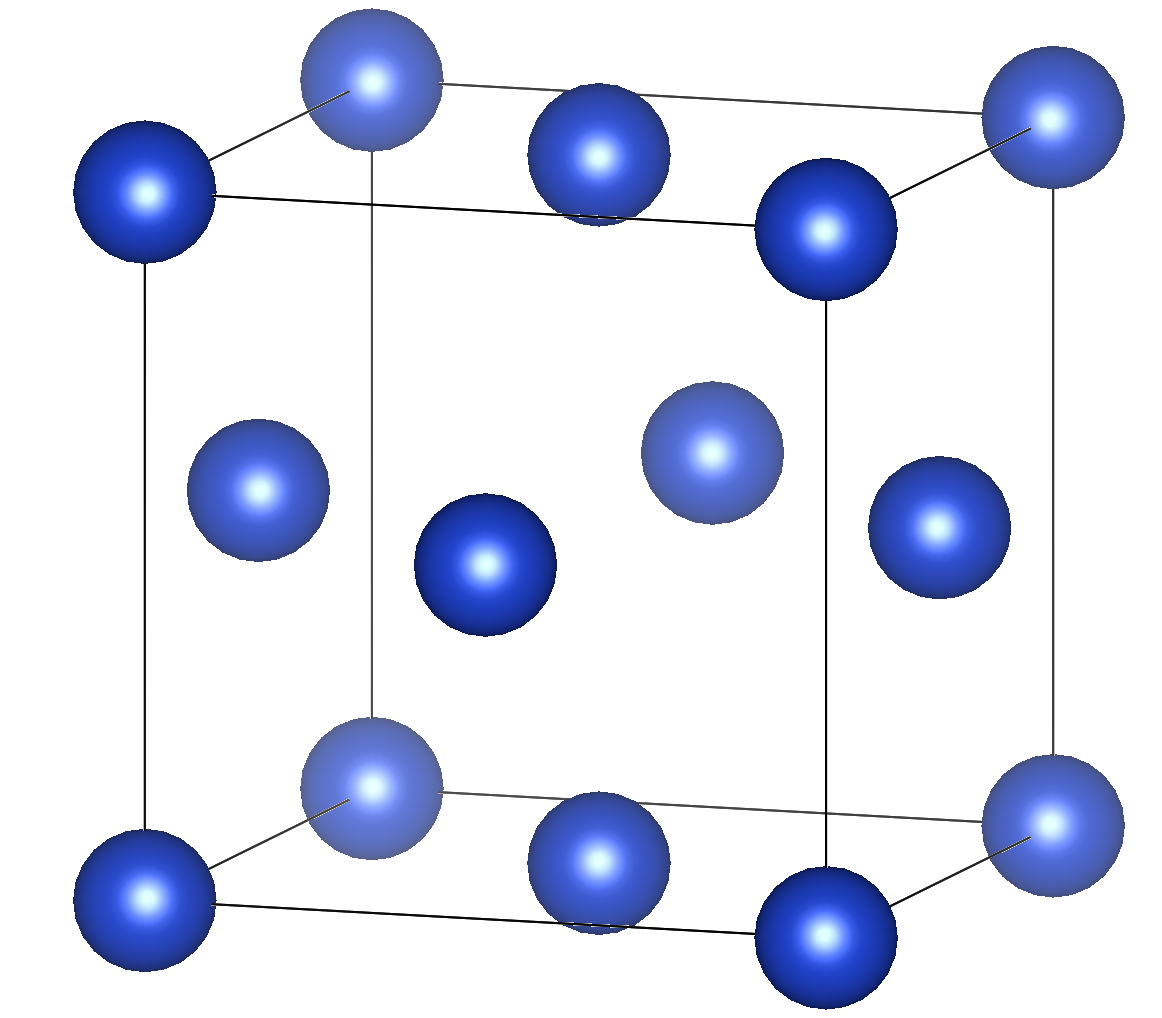}
\hskip 0.3cm
\includegraphics[width=5.0cm]{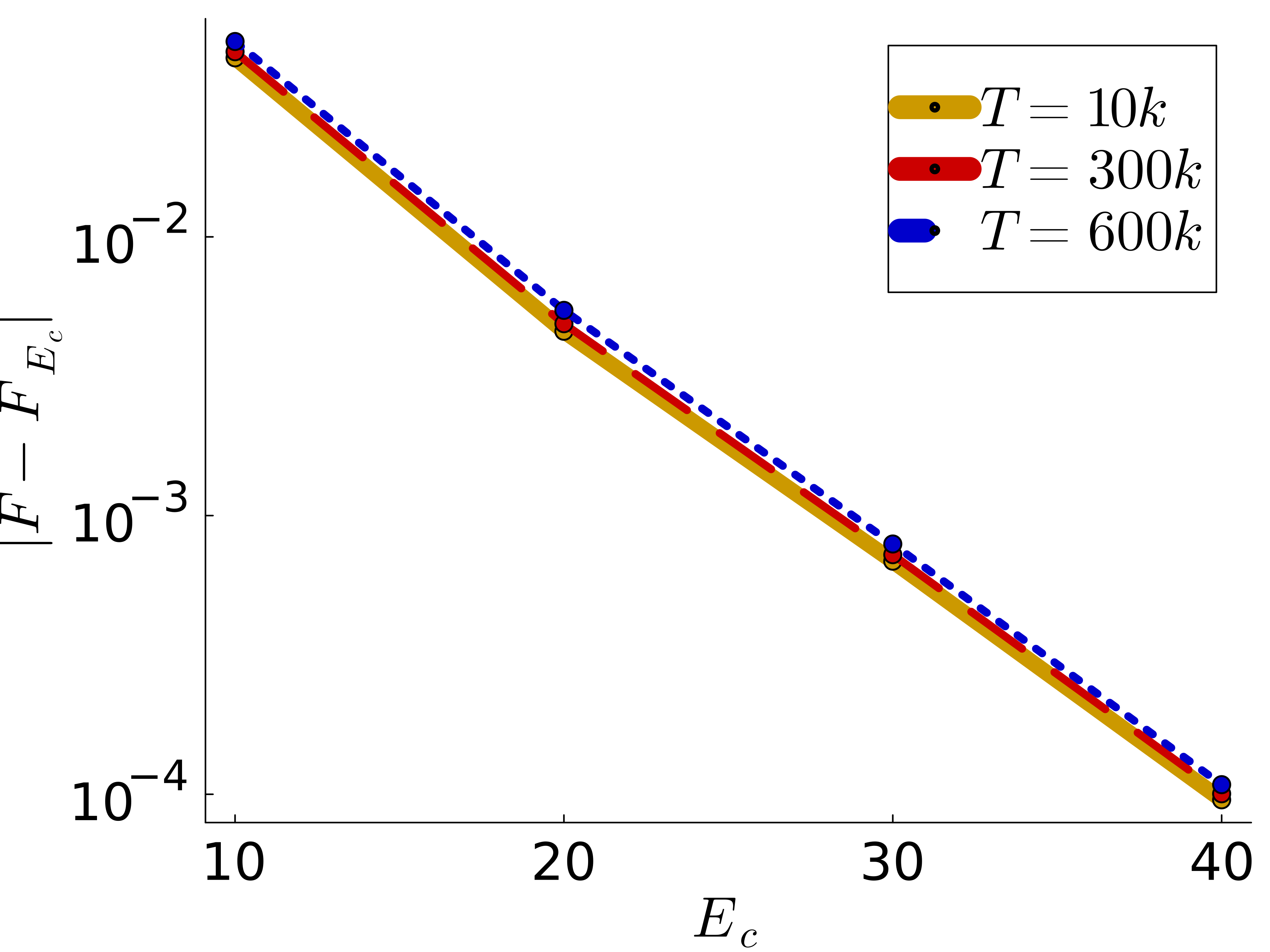}
\hskip 0.3cm
\includegraphics[width=5.0cm]{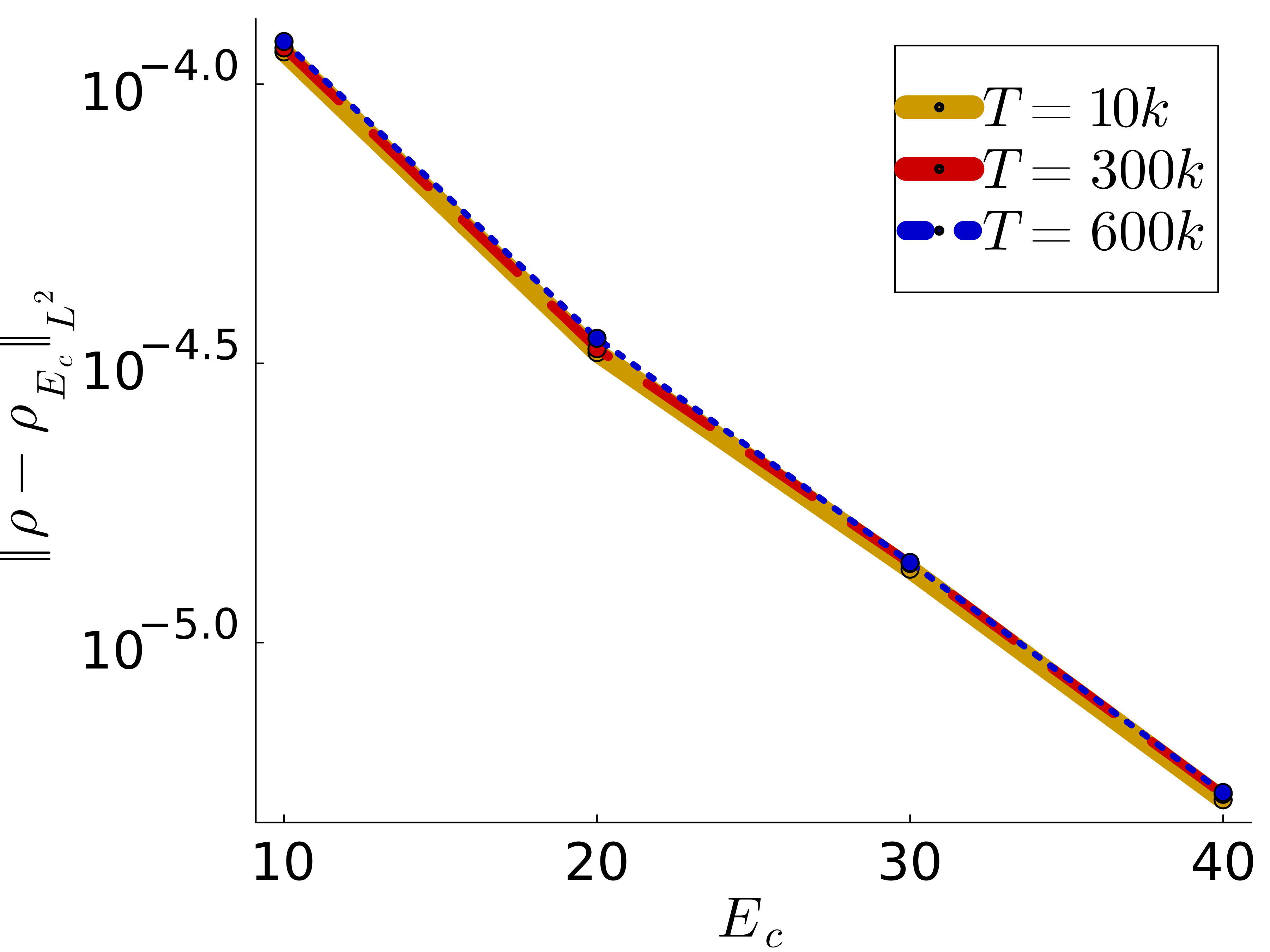}
\caption{(Example 2. Aluminium) 
Left: FCC lattice in a unit cell. Middle: energy error with respect to the energy cutoffs. Right: $L^2$-error of the electron density with respect to the energy cutoffs.}
\label{fig:Al}
\vskip 0.2cm
\end{figure}

\section{Concluding remarks}
\label{sec_conclusion}
\setcounter{equation}{0}

In this work, we study the numerical aspects of the ground state calculations for Mermin-Kohn-Sham equation in finite temperature DFT.
We reformulate the nonlinear eigenvalue problem by a variational problem with respect to the one-electron density matrix, then justify the convergence of finite dimensional approximations and further derive an {\it a priori} error estimate.
For future perspective, we will study the zero temperature limit of the problem, which is more complicated as the stability is lost at this limit for metallic systems.

\subsection*{Acknowledgments}

The authors would like to thank Reinhold Schneider and Christoph Ortner for inspiring discussions on this topic.
This work was supported by the National Key R \& D Program of China under grants 2019YFA0709600 and 2019YFA0709601. 
HC's work was also supported by National Natural Science Foundation of China under grant NSFC12371431.
XG's work was supported by National Natural Science Foundation of China under grants U23A20537 and Funding of National Key Laboratory of Computational Physics.

\small
\bibliographystyle{plain}
\bibliography{bib.bib}

\end{document}

%% file: pre.tex
\usepackage{mathrsfs}
\usepackage{graphicx}
\usepackage{amsmath}
\usepackage{amsfonts}
\usepackage{amssymb}
\usepackage{amsthm}
\usepackage{textcomp}
\usepackage{color}
\usepackage{geometry}
\usepackage[normalem]{ulem}
\usepackage{placeins}
\usepackage{bm}
\usepackage{wasysym}
\usepackage{titlesec}
\usepackage{subcaption}
\newtheorem{theorem}{Theorem}[section]
\newtheorem{lemma}{Lemma}[section]
\newtheorem{remark}{Remark}[section]

\def\R{\mathbb{R}}
\def\E{\mathcal{E}}
\def\F{\mathcal{F}}
\def\G{\mathcal{G}}
\def\Gp{\pmb{G}}
\def\rr{\pmb{r}}
\def\J{\mathcal{J}}

\def\L{\mathcal{L}}
\def\LL{\mathbb{L}}
\def\dd{~{\rm d}}
\def\Ec{E_{\rm c}}

\def\<{\langle}
\def\>{\rangle}

